\documentclass[a4paper,12pt]{article}
\usepackage[T2A]{fontenc} 

\usepackage[english]{babel}
\usepackage{mathrsfs}
\usepackage[top=2cm,left=3cm,right=1.5cm,bottom=2cm]{geometry}
\usepackage{amsmath}
\usepackage{amsfonts}
\usepackage{amssymb}
\usepackage{amsthm} 
\usepackage{enumerate}

\numberwithin{equation}{subsection}

\theoremstyle{plain}
\newtheorem{thm}{Theorem}[section]
\newtheorem{cor}[thm]{Corollary}

\theoremstyle{definition}
\newtheorem{defn}[thm]{Definition}

\theoremstyle{remark}
\newtheorem{rem}[thm]{Remark}

\numberwithin{equation}{section}

\begin{document}


\begin{center}\textbf{On the absence of Volterra correct restrictions and extensions\\
of the Laplace operator}\end{center}

\begin{center}\textbf{Bazarkan\,N. Biyarov}\end{center}

\begin{center}January 30, 2016\end{center}
\\


\begin{abstract}
At the beginning of the last century J.\,Hadamard constructed the well-known example illustrating the incorrectness of the Cauchy problem for elliptic-type equations. If the Cauchy problem for some differential equation is correct, then it is usually a Volterra problem, i.e., the inverse operator is a Volterra operator. At present, not a single Volterra correct restriction or extension for elliptic-type equations is known.
In the present paper, we prove the absence of Volterra correct restrictions of the maximal operator $\widehat{L}$ and Volterra correct extensions of the minimal operator $L_0$ generated by the Laplace operator in $L_2(\Omega)$, where $\Omega$ is the unit disk.
\end{abstract}

\section{Introduction}
\label{section1}

Let us present some definitions, notation, and terminology.

In a Hilbert space $H$, we consider a linear operator $L$ with  domain $D(L)$ and range $R(L)$. By the \textit{kernel} of the operator $L$ we mean the set
\[\mbox{Ker}\,L=\bigl\{f\in D(L): \; Lf=0\bigl\}.\]

\begin{defn}
\label{def1}  
An operator $L$ is called a \textit{restriction} of an operator $L_1$, and $L_1$ is called an \textit{extension} of an operator $L$, briefly $L\subset L_1$, if:

1) $D(L)\subset D(L_1)$,

2) $Lf=L_1f$ for all $f$ from $D(L)$.
\end{defn}

\begin{defn}
\label{def2}
A linear closed operator $L_0$ in a Hilbert space $H$ is called \textit{minimal} if $\overline{R(L_0)} \not=H$ and there exists a bounded inverse operator $L_0^{-1}$ on $R(L_0)$. 
\end{defn}

\begin{defn}
\label{def3}
A linear closed operator $\widehat{L}$ in a Hilbert space $H$ is called \textit{maximal} if $R(\widehat{L})=H$ and $\mbox{Ker}\, \widehat{L} \not=\{0\}$. 
\end{defn}

\begin{defn}
\label{def4}
A linear closed operator $L$ in a Hilbert space $H$ is called \textit{correct} if there exists a bounded inverse operator $L^{-1}$ defined on all of $H$. 
\end{defn}

\begin{defn}
\label{def5}
We say that a correct operator $L$ in a Hilbert space $H$ is a \textit{correct extension} of minimal operator $L_0$ (\textit{correct restriction} of maximal operator $\widehat{L}$) if $L_0\subset L$ ($L\subset \widehat{L}$).
\end{defn}

\begin{defn}
\label{def6}
We say that a correct operator $L$ in a Hilbert space $H$ is a \textit{boundary correct} extension of a minimal operator $L_0$ with respect to a maximal operator $\widehat{L}$ if $L$ is simultaneously a correct restriction of the maximal operator $\widehat{L}$ and a correct extension of the minimal operator $L_0$, that is, $L_0\subset L \subset \widehat{L}$.
\end{defn}

Let $\widehat{L}$ be a maximal linear operator in a Hilbert space $H$, let $L$ be any known correct restriction of $\widehat{L}$, and let $K$ be an arbitrary linear bounded (in $H$) operator satisfying the following condition:
\begin{equation}\label{1.1}
R(K)\subset \mbox{Ker}\, \widehat{L}.
\end{equation}
Then the operator $L_K^{-1}$ defined by the formula (see {\cite{Kokebaev}})
\begin{equation}\label{1.2}
L_K^{-1}f=L^{-1}f+Kf,
\end{equation}
describes the inverse operators to all possible correct restrictions $L_K$ of $\widehat{L}$, i.e., $L_K\subset \widehat{L}$.

Let $L_0$ be a minimal operator in a Hilbert space $H$, let $L$ be any known correct extension of $L_0$, and let $K$ be a linear bounded operator in $H$ satisfying the conditions

a) $R(L_0)\subset \mbox{Ker}\,K$,

b) $\mbox{Ker}\,(L^{-1}+K)=\{0\}$,
\\
then the operator $L_K^{-1}$ defined by formula \eqref{1.2}
describes the inverse operators to all possible correct extensions $L_K$ of  $L_0$ (see {\cite{Kokebaev}}).

Let $L$ be any known boundary correct extension of $L_0$, i.e., $L_0\subset L\subset \widehat{L}$. The existence of at least one boundary correct extension $L$ was proved by Vishik in {\cite{Vishik}}. Let $K$ be a linear bounded (in $H$) operator satisfying the conditions

a) $R(L_0)\subset \mbox{Ker}\,K$,

b) $R(K)\subset \mbox{Ker}\,\widehat{L}$,
\\
then the operator $L_K^{-1}$ defined by formula \eqref{1.2}
describes the inverse operators to all possible boundary correct extensions $L_K$ of $L_0$ (see {\cite{Kokebaev}}). 

\begin{defn}
\label{def7} A bounded operator $A$ in a Hilbert space $H$ is called \textit{quasinilpotent} if its spectral
radius is zero, that is, the spectrum consists of the single point zero.
\end{defn}

\begin{defn}
\label{def8} An operator $A$ in a Hilbert space $H$ is called a \textit{Volterra operator} if $A$ is compact and quasinilpotent.
\end{defn}

\begin{defn}
\label{def9} A correct restriction $L$ of a maximal operator $\widehat{L} \;(L\subset \widehat{L})$, a correct extension $L$ of a minimal operator $L_0 \; (L_0 \subset L)$ or a boundary correct extension $L$ of a minimal operator $L_0$ with respect to a maximal operator $\widehat{L} \; (L_0\subset L \subset \widehat{L})$, will be called \textit{Volterra} if the inverse operator $L^{-1}$ is a
Volterra operator.
\end{defn}

We denote by $\mathfrak{S}_{\infty}(H, H_{1})$ the set of all linear compact operators acting from a Hilbert space $H$ to a
Hilbert space $H_1$. If $T\in \mathfrak{S}_{\infty}(H, H_{1})$, then $T^{*}T$ is a non-negative self-adjoint operator in $\mathfrak{S}_{\infty}(H)\equiv \mathfrak{S}_{\infty}(H, H)$ and, moreover, there is a non-negative unique self-adjoint root $|T|=(T^{*}T)^{1/2}$ in $\mathfrak{S}_{\infty}(H)$. The eigenvalues $\lambda _n(|T|)$ numbered, taking into account their multiplicity, form a monotonically converging to zero sequence of non-negative numbers. These numbers are usually called \textit{$s$-numbers} of the operator $T$ and denoted by $s_n(T)$, $n\in \mathbb{N}$. We denote by $\mathfrak{S}_{p}(H, H_{1})$ the set of all compact operators $T\in \mathfrak{S}_{\infty}(H, H_{1})$, for which
\[|T|_p^p=\sum_{j=1}^{\infty} s_{j}^{p}(T)<\infty, \quad 0<p<\infty.\]

Obviously, if rank $\mbox{rank}\,|T|=r< \infty$, then $s_n(T)=0$, for $n=r+1, r+2, \ldots\:$. Operators of finite rank certainly belong to the classes $\mathfrak{S}_{p}(H, H_{1})$ for all $p>0$.

In the Hilbert space $L_2(\Omega)$, where $\Omega$ is the unit disk in $\mathbb R^2$ with  boundary $\partial\Omega$, let us consider the minimal $L_0$ and maximal $\widehat{L}$ operators generated by the Laplace operator
\begin{equation}\label{1.3}
-\Delta u=-\biggl(\frac{\partial^2 u}{\partial{x^2}}+\frac{\partial^2 u}{\partial{y^2}}\biggr).
\end{equation}

The closure $L_0$ in the space $L_2(\Omega)$ of the Laplace operator \eqref{1.3} with the domain $C_0^\infty (\Omega)$ is the \textit{minimal operator corresponding to the Laplace operator}.

The operator $\widehat{L}$, adjoint to the minimal operator $L_0$ corresponding to the Laplace operator is the \textit{maximal operator corresponding to the Laplace operator} (see \cite{Hormander}). Note that \[D(\widehat{L})=\{u\in L_2(\Omega): \; \widehat{L}u =-\Delta u \in L_2(\Omega)\}.\]

Denote by $L_D$ the operator, corresponding to the Dirichlet problem with the domain  
\[ D(L_D)=\{u\in W_2^2(\Omega): \; u|_{\partial\Omega}=0\}. \]

Then, by virtue of \eqref{1.2}, the inverse operators $L^{-1}$ to all possible correct restrictions of the maximal operator $\widehat{L}$ corresponding to the Laplace operator \eqref{1.3} have the following form:
\begin{equation}
u\equiv L^{-1}f=L_D^{-1}f+Kf, \label{1.4}
\end{equation}
where, by virtue of \eqref{1.1}, $K$ is an arbitrary linear operator bounded in $L_2(\Omega)$ with 
\[R(K)\subset \mbox{Ker}\,\widehat{L} =\{u\in L_2(\Omega): \: -\Delta u=0\}.\]

Then the direct operator $L$ is determined from the following problem:
\begin{equation}
\widehat{L}u\equiv -\Delta u=f, \quad f\in L_2(\Omega),\label{1.5}
\end{equation}
\begin{equation}
D(L)=\{ u\in D(\widehat{L}) : \; [(I-K \widehat{L})u]|_{\partial\Omega}=0 \}, \label{1.6}
\end{equation}
where $I$ is the unit operator in $L_2(\Omega)$. There are no other linear correct restrictions of the operator $\widehat{L}$ (see \cite{Biyarov}).

The operators $(L^*)^{-1}$, corresponding to the adjoint operators $L^*$
\[
v\equiv (L^*)^{-1}g=L_D^{-1}g+K^*g,
\]
describe the inverse operators to all possible correct extensions of the minimal operator $L_0$ if and only if $K$ satisfies the condition (see \cite{Biyarov}):
\[ \mbox{Ker}(L_D^{-1}+K^*)=\{0\}. \]
Note that the last condition is equivalent to the following: $\overline{D(L)}= L_2(\Omega)$.
If the operator $K$ from \eqref{1.4} satisfies one more additional condition
\[ KR(L_0)=\{0\}, \]
then the operator $L$ corresponding to problem \eqref{1.5}, \eqref{1.6}, will turn out to be a boundary correct extension. 

Now we state the main result.

\section{Main results}
\label{section2}

We pass to the polar coordinate system:
\[x=r \cos\varphi, \: y=r \sin\varphi.\]
Then the operator
\begin{equation} 
\label{2.1} 
\widehat{L}u \equiv -\Delta u=-\frac{\partial^2 u}{\partial{x^2}}-\frac{\partial^2 u}{\partial{y^2}}=-u_{rr}-\frac{1}{r}u_r-\frac{1}{r^2}u_{\varphi\varphi}=f(r, \varphi), 
\end{equation}
on
\[D(\widehat{L})=\{u\in L_2(\Omega): \; \Delta u \in L_2(\Omega)\},\]
is the maximal operator (see \cite{Biyarov2}). Any correct restriction $L$ acts as the maximal operator $\widehat{L}$ on the domain
\begin{equation} 
\label{2.2} 
D(L)=\{u\in D(\widehat{L}): \; [(I-K \widehat{L})u]|_{r=1}=0  \}, 
\end{equation}
where $K$ is any bounded linear operator in $L_2(\Omega)$ that
$R(K)\subset \mbox{Ker}\,\widehat{L}$.
To be Volterra of $L$ is necessary compactness of $L^{-1}$.
From \eqref{1.4} note that $L^{-1}$ is compact if and only if $K$ is a compact operator. Then for $K$ the Schmidt expansion takes place (see \cite[p.\,47(28)]{Gohberg})
\begin{equation}\label{2.3}
K=\sum_{j=1}^{\infty}s_j(\:\cdot\,, \,Q_j)F_j               
\end{equation}       
where $\{Q_j \}_1^\infty$ is orthonormal system in $L_2(\Omega)$, $\,\{F_j \}_1^\infty$ is orthonormal system in $\mbox{Ker}\,\widehat{L}\,$ and $\,\{s_j \}_1^\infty$ is a monotone sequence of non-negative numbers converging to zero. The series on the right side of \eqref{2.3} converges in the uniform operator norm. We now state the main result of this paper.
\begin{thm}\label{Theorem1}
Let $\widehat{L}$ be a maximal operator generated by the Laplace~\eqref{1.3} in $L_2(\Omega)$. Then any correct restriction $L$ of the maximal operator $\widehat{L}$, i.e., the problem \eqref{2.1} and \eqref{2.2} cannot be Volterra. 
\end{thm}
\begin{proof}
Let us prove by contradiction. Suppose that there exists a Volterra correct restriction $L$. This is equivalent to the existence of a such compact operator $K$ that the operator $L$ has no non-zero eigenvalue. The general solution of the equation
\[\widehat{L}u=-\Delta u=-u_{rr}-\frac{1}{r}u_r-\frac{1}{r^2}u_{\varphi\varphi}=\lambda^2u,\]
from the space $L_2(\Omega)$ has the form (see \cite{Vekya}) 
\[u(r,\varphi)=u_0(r,\varphi)-\int_0^r u_0(\rho,\varphi)\frac{\partial}{\partial\rho} J_0(\lambda\sqrt{r(r-\rho)})d\rho,\]
where $\lambda$ is any complex number, $u_0(r,\varphi)$ is the solution of the equation
\[\widehat{L}u_0\equiv-\Delta u_0=-u_{0_{rr}}-\frac{1}{r}u_{0_r}-\frac{1}{r^2}u_{0_{\varphi\varphi}}=0,\]
which is a harmonic function from the space $L_2(\Omega)$ and 
\[J_0(z)=\sum_{n=0}^\infty\dfrac{(-1)^n}{(n!)^2}\Bigl(\dfrac{z}{2}\Bigr)^{2n}\]  
is the Bessel function.
Then, by virtue of \eqref{2.2} we obtain the equation
\begin{equation}\label{2.4}
\begin{split}
&u_0(1,\varphi)-\int_0^1 u_0(\rho,\varphi)\frac{\partial}{\partial\rho}J_0(\lambda\sqrt{1-\rho})d\rho \\[5pt]
&\, -\lambda^2 \sum_{j=1}^{\infty}s_j F_j(1, \varphi)\cdot \frac{1}{2\pi}\int_0^{2\pi}\int_0^1u_0(\rho,\theta)\cdot \overline{Q_j(\rho,\theta)} \rho d\rho d\theta \\[5pt]
&\, +\lambda^2 \sum_{j=1}^{\infty}s_j F_j(1, \varphi)\frac{1}{2\pi}\int_0^{2\pi}\int_0^1 \overline{Q_j(\rho,\theta)}\cdot \int_0^\rho u_0(\tau,\theta)\frac{\partial}{\partial\tau}J_0(\lambda \sqrt{\rho(\rho-\tau)}d\tau\rho d\rho d\theta =0.
\end{split}
\end{equation}

The considered problem on the spectrum of the Laplace operator has no eigenvalues if and only if the equation \eqref{2.4} has no zeros as a function of $\lambda$. The harmonic function $u_0(\rho,\varphi)$ does not depend on $\lambda$. It is easy to notice that the left side of the equation is an entire function no higher than the first order. Then by virtue of Picard's theorem (see \cite[p.\,264, 266]{Markushevich}) this function have the form $C e^{d\lambda}$, where $C(\varphi)$ and $d(\varphi)$ are a functions which are independent of $\lambda$. If you notice that the left side of the equation \eqref{2.4} is even with respect to the sign of $\lambda$, then $d=0$. Equating these functions when $\lambda=0$ we have $C=u_0(1,\varphi)$. Then we get the following
\begin{equation}\label{2.5}
\begin{split}
&-\int_0^1 u_0(\rho,\varphi)\frac{\partial}{\partial\rho}J_0(\lambda\sqrt{1-\rho})d\rho \\[5pt]
&\, -\lambda^2 \sum_{j=1}^{\infty}s_j F_j(1, \varphi)\cdot \frac{1}{2\pi}\int_0^{2\pi}\int_0^1u_0(\rho,\theta)\cdot \overline{Q_j(\rho,\theta)} \rho d\rho d\theta \\[5pt]
&\quad +\lambda^2 \sum_{j=1}^{\infty}s_j F_j(1, \varphi)\cdot \frac{1}{2\pi}\int_0^{2\pi}\int_0^1 \overline{Q_j(\rho,\theta)} \int_0^\rho  u_0(\tau,\theta)\frac{\partial}{\partial\tau}J_0(\lambda \sqrt{\rho(\rho-\tau)}d\tau\rho d\rho d\theta =0.
\end{split}
\end{equation}
Divide both sides of \eqref{2.5} by $\lambda^2$ and let $\lambda$ tend to zero. Then
\begin{equation}\label{2.6}
\sum_{j=1}^{\infty}s_j F_j(1, \varphi)\cdot \frac{1}{2\pi}\int_0^{2\pi}\int_0^1 u_0(\rho,\theta) \overline{Q_j(\rho,\theta)}\rho d\rho d\theta= -\frac{1}{4}\int_0^1 u_0(\rho,\varphi)d\rho.
\end{equation}
Under the condition that \eqref{2.6} is fulfilled we obtain
\begin{equation}\label{2.7}
\begin{split}
&-\int_0^1 u_0(\rho,\varphi)\biggl[\frac{\partial J_0}{\partial\rho}(\lambda\sqrt{1-\rho})+\frac{1}{4}\biggr]d\rho \\[5pt]
&\, + \sum_{j=1}^{\infty}s_j F_j(1, \varphi)\cdot \frac{1}{2\pi}\int_0^{2\pi}\int_0^1 \overline{Q_j(\rho,\theta)} \int_0^\rho  u_0(\tau,\theta)\frac{\partial}{\partial\tau} J_0(\lambda \sqrt{\rho(\rho-\tau)}d\tau\rho d\rho d\theta =0.
\end{split}
\end{equation}
On the left side of the equation \eqref{2.7} we make the change of variables: in the first summand $t=\sqrt{1-\rho}$,  in the second summand $t=\sqrt{\rho(\rho-\tau)}$. Then we have
\begin{equation}\label{2.8}
\begin{split}
&\int_0^1 u_0(1-t^2,\varphi)\biggl[\frac{J_0'(\lambda t)}{2\lambda t}+\frac{1}{4}\biggr] 2tdt \\[5pt]
&\, - \sum_{j=1}^{\infty}s_j F_j(1, \varphi)\cdot \frac{1}{2\pi}\int_0^{2\pi}\int_0^1 J_0'(\lambda t) \int_t^1 u_0\Bigl(\frac{\rho^2-t^2}{\rho}, \theta\Bigr) \overline{Q_j(\rho,\theta)} \rho d\rho dt d\theta =0.
\end{split}
\end{equation}
For the Bessel function has the following equalities
\[\dfrac{J'_0(\lambda t)}{2\lambda t}+\dfrac{1}{4}=\dfrac{1}{4}\sum_{n=1}^\infty \dfrac{(-1)^{n+1}}{((n+1)!)^2}\Bigl(\dfrac{\lambda t}{2}\Bigr)^{2n}\cdot(n+1),\]
and
\[\lambda J'_0(\lambda t)=\sum_{n=1}^\infty \dfrac{(-1)^{n}}{(n!)^2}\Bigl(\dfrac{\lambda t}{2}\Bigr)^{2n}\cdot n\cdot\dfrac{2}{t}.\]
Substitute them into \eqref{2.8} and equate the coefficients of $\lambda^{2n}$ to zero
\[
\begin{split}
&\int_0^1 u_0(1-t^2,\varphi)\frac{-1}{4(n+1)}\cdot t^{2n}\cdot 2tdt \\[5pt]
&\, - \sum_{j=1}^{\infty}s_j F_j(1, \varphi)\cdot \frac{1}{2\pi}\int_0^{2\pi}\int_0^1 nt^{2n}\cdot \frac{2}{t} \int_t^1 u_0\Bigl(\frac{\rho^2-t^2}{\rho}, \theta\Bigr) \overline{Q_j(\rho,\theta)} \rho d\rho dt d\theta =0.
\end{split}
\]
We do the conversion of the following form
\[\frac{1}{n+1}\cdot t^{2n}=\frac{2}{t^2}\int_0^t\tau^{2n+1}d\tau,\qquad n\cdot t^{2n}=\frac{t}{2}\cdot \frac{\partial}{\partial t}\Bigl(t^{2n}\Bigr).\]
Then
\[\begin{split}
&\int_0^1t^{2n}\cdot \biggl\{ t\int_t^1 u_0(1-\tau^2,\varphi) \frac{d\tau}{\tau}dt\\[5pt]
&\quad-\frac{1}{2\pi}\int_0^{2\pi} \frac{\partial}{\partial t} \sum_{j=1}^{\infty}s_j F_j(1, \varphi)\cdot \int_t^1 u_0\Bigl(\frac{\rho^2-t^2}{\rho}, \theta\Bigr) \overline{Q_j(\rho,\theta)} \rho d\rho d\theta \biggr\} dt=0.
\end{split}\]
In view of the completeness of the system of functions $\{t^{2n}\}_1^\infty$ in $L_2(0,1)$ we obtain (see \cite[p.\,107]{Kach})
\[
 t\int_t^1 u_0(1-\tau^2,\varphi) \frac{d\tau}{\tau}
-\frac{1}{2\pi}\int_0^{2\pi} \frac{\partial}{\partial t} \sum_{j=1}^{\infty}s_j F_j(1, \varphi)\cdot \int_t^1 u_0\Bigl(\frac{\rho^2-t^2}{\rho}, \theta\Bigr) \overline{Q_j(\rho,\theta)} \rho d\rho d\theta=0.
\]
Integrating this equation from $t$ to $1$, we get
\begin{equation}\label{2.9}
\begin{split}
&\,\int_t^1 u_0(1-\tau^2,\varphi) \frac{\tau^2-t^2}{2\tau}d\tau\\[5pt]
&\quad+\frac{1}{2\pi}\int_0^{2\pi} \int_t^1 u_0\Bigl(\frac{\rho^2-t^2}{\rho}, \theta\Bigr) \sum_{j=1}^{\infty}s_j F_j(1, \varphi) \overline{Q_j(\rho,\theta)} \rho d\rho d\theta=0.
\end{split}
\end{equation}
where $\,0\leq t\leq1, \; 0\leq\varphi<2\pi$. Note that the condition \eqref{2.9} contains the condition \eqref{2.6} as a particular case when $t=0$. Condition \eqref{2.9} will turn out to be the Volterra criterion of the correct restriction $L$, if it holds for any harmonic function $u_0(r, \varphi)$ from $L_2(\Omega)$.

By Poisson's formula
\[u_0(r,\varphi)=\frac{1}{2\pi}\int_0^{2\pi}\dfrac{1-r^2}{1-2r\cos(\varphi-\gamma)+r^2}u_0(1,\gamma)d\gamma,\]
the equality \eqref{2.9} is transformed to
\[
\begin{split}
&\frac{1}{2\pi}\int_0^{2\pi}u_0(1,\gamma)\biggl\{\int_t^1\dfrac{1-(1-\tau^2)^2}{1-2(1-\tau^2)\cos(\varphi-\gamma)+(1-\tau^2)^2}\cdot \dfrac{\tau^2-t^2}{2\tau}d\tau\\[5pt]
&\,+\frac{1}{2\pi}\int_0^{2\pi} \int_t^1\dfrac{1-\bigl(\frac{\rho^2-t^2}{\rho}\bigr)^2}{1-2\bigl(\frac{\rho^2-t^2}{\rho}\bigr)\cos(\theta-\gamma)+\bigl(\frac{\rho^2-t^2}{\rho}\bigr)^2}\sum_{j=1}^{\infty}s_j F_j(1, \varphi) \overline{Q_j(\rho,\theta)}\rho d\rho d\theta\biggr\}d\gamma=0.
\end{split}
\]
Considering the density of the set of functions $u_0(1, \varphi)$ in $L_2(0, 2\pi)$ for almost all values of $t \, (0\leq t\leq1), \; \varphi \; (0\leq\varphi<2\pi)$ we obtain the equality
\begin{equation}\label{2.10}
\begin{split}
&\int_t^1\dfrac{1-(1-\tau^2)^2}{1-2(1-\tau^2)\cos(\varphi-\gamma)+(1-\tau^2)^2}\cdot \dfrac{\tau^2-t^2}{2\tau}d\tau\\[5pt]
&\,+\frac{1}{2\pi}\int_0^{2\pi} \int_t^1\dfrac{1-\bigl(\frac{\rho^2-t^2}{\rho}\bigr)^2}{1-2\bigl(\frac{\rho^2-t^2}{\rho}\bigr)\cos(\theta-\gamma)+\bigl(\frac{\rho^2-t^2}{\rho}\bigr)^2}\sum_{j=1}^{\infty}s_j F_j(1, \varphi) \overline{Q_j(\rho,\theta)}\rho d\rho d\theta=0.
\end{split}
\end{equation}
Now the equation \eqref{2.10} is the Volterra criterion of the correct restriction $L$ of the maximal operator $\widehat{L}$ generated by the Laplace operator \eqref{1.3} in $L_2(\Omega)$, where $\Omega$ is the unit disk. 

Further, we apply to the equation \eqref{2.10} the Poisson operator of the variables $r$ and $\varphi$. The first summand we transform with the formula of the superposition of two Poisson integrals (see {\cite[p. 140]{Axler}}), and in the second summand the harmonic function $F_j(r,\varphi)$ is reproduced by Poisson's formula. We have
\[
\begin{split}
&\int_t^1\dfrac{1-r^2(1-\tau^2)^2}{1-2r(1-\tau^2)\cos(\varphi-\gamma)+r^2(1-\tau^2)^2}\cdot \dfrac{\tau^2-t^2}{2\tau}d\tau\\[5pt]
&\,+\frac{1}{2\pi}\int_0^{2\pi} \int_t^1\dfrac{1-\bigl(\frac{\rho^2-t^2}{\rho}\bigr)^2}{1-2\bigl(\frac{\rho^2-t^2}{\rho}\bigr)\cos(\theta-\gamma)+\bigl(\frac{\rho^2-t^2}{\rho}\bigr)^2}\sum_{j=1}^{\infty}s_j F_j(r, \varphi) \overline{Q_j(\rho,\theta)}\rho d\rho d\theta=0.
\end{split}
\]
From this equality using the orthonormality of the system $\{F_j(r,\varphi)\}_1^\infty$ we obtain the relation between the orthonormal systems $\{F_j\}_1^\infty$ and $\{Q_j\}_1^\infty$ of the following form
\begin{equation}\label{2.11}
\begin{split}
&\int_t^1 \biggl\{\frac{1}{2\pi}\int_0^{2\pi}\int_0^1\dfrac{1-r^2(1-\tau^2)^2}{1-2r(1-\tau^2)\cos(\varphi-\gamma)+r^2(1-\tau^2)^2} \overline{F_j(r,\varphi)}r drd\varphi\biggr\}\cdot \dfrac{\tau^2-t^2}{2\tau}d\tau\\[5pt]
&\,=-\frac{1}{2\pi}\int_0^{2\pi} \int_t^1\dfrac{1-\bigl(\frac{\rho^2-t^2}{\rho}\bigr)^2}{1-2\bigl(\frac{\rho^2-t^2}{\rho}\bigr)\cos(\theta-\gamma)+\bigl(\frac{\rho^2-t^2}{\rho}\bigr)^2}\cdot s_j \overline{Q_j(\rho,\theta)}\rho d\rho d\theta, \quad j=1, 2, \ldots
\end{split}
\end{equation}
In both parts of the equality \eqref{2.11} we use the expansion of the Poisson kernel
\[
\frac{1-r^2}{1-2r\cos \varphi+r^2}=1+2\sum_{n=1}^\infty r^n\cos n\varphi.
\]
We obtain the equality of the two Fourier series in the orthogonal system $\{1/2, \cos  \gamma, \sin  \gamma, \ldots,$ $\cos n\gamma, \sin n\gamma, \ldots\}$ in $L_2(0, 2\pi)$. Equating the coefficients, we get the following system of equations
\begin{equation}\label{2.12}
\left \{\begin{split}
&\frac{1}{2\pi}\int_0^{2\pi}\int_0^1 \overline{F_j(r,\varphi)}\cdot rdrd\varphi\cdot \int_t^1\dfrac{\tau^2-t^2}{2\tau}d\tau=-\frac{1}{2\pi}\int_0^{2\pi}\int_t^1 s_j \overline{Q_j(\rho,\theta)}\rho d\rho d\theta, \\[8pt]
&\frac{1}{\pi}\int_0^{2\pi}\int_0^1 \overline{F_j(r,\varphi)}\cdot r^{n+1}\cos n\varphi drd\varphi\int_t^1 (1-\tau^2)^n\dfrac{\tau^2-t^2}{2\tau}d\tau \\[4pt]
&\quad=-\int_t^1\frac{1}{\pi}\int_0^{2\pi}s_j \overline{Q_j(\rho,\theta)}\cdot \cos n\theta d\theta \biggl(\dfrac{\rho^2-t^2}{\rho}\biggr)^n\rho d\rho ,\\[8pt]
&\frac{1}{\pi}\int_0^{2\pi}\int_0^1 \overline{F_j(r,\varphi)}\cdot r^{n+1}\cdot\sin n\varphi drd\varphi\int_t^1 (1-\tau^2)^n\dfrac{\tau^2-t^2}{2\tau}d\tau \\[4pt]
&\quad=-\int_t^1\frac{1}{\pi}\int_0^{2\pi}s_j \overline{Q_j(\rho,\theta)}\cdot \sin n\theta d\theta \biggl(\dfrac{\rho^2-t^2}{\rho}\biggr)^n\rho d\rho, \quad j=1, 2, \ldots, \quad n=1, 2, \ldots.
\end{split} \right.
\end{equation}
We denote
\[
A_{nj}=\frac{1}{\pi}\int_0^{2\pi}\int_0^1 \overline{F_j(r,\varphi)}\cos n\varphi\cdot r^n\cdot rdrd\varphi,\quad n=0, 1, 2,\ldots
\]
\[
B_{nj}=\frac{1}{\pi}\int_0^{2\pi}\int_0^1 \overline{F_j(r,\varphi)}\sin n\varphi\cdot r^n\cdot rdrd\varphi,\quad n=1, 2,\ldots
\]
From the first equation of the system \eqref{2.12} it is easy to find that
\[
\frac{1}{2\pi}\int_0^{2\pi}s_j \overline{Q_j(t,\theta)}d\theta=\frac{1}{2}A_{0j}\cdot lnt.
\]
The second equation reduces to
\begin{equation}\label{2.13}
\int_t^1(1-\tau^2)^n\dfrac{\tau^2-t^2}{2\tau}d\tau=-\int_t^1(\rho^2-t^2)^n\omega_n(\rho)\rho d\rho,\quad n=1,2,\ldots
\end{equation}
if we denote by
\[
\omega_n(\rho)=\frac{1}{\pi}\int_0^{2\pi}s_j \overline{Q_j(\rho,\theta)}\cos n\theta d\theta\dfrac{1}{A_{nj}\rho^n}.
\]
The third equation is transformed into the same equation \eqref{2.13}, if we denote by
\[
\omega_n(\rho)=\frac{1}{\pi}\int_0^{2\pi}s_j \overline{Q_j(\rho,\theta)}\sin n\theta d\theta\dfrac{1}{B_{nj}\rho^n}.
\]
We solve the equation \eqref{2.13} with respect to $\omega_n(\rho)$. Note that
\[\omega_1(t)=-\dfrac{1-t^2}{2t^2},\quad \omega_2(t)=-\dfrac{1-t^4}{4t^4};\]
Further, we get the recurrence relation
\[
\begin{split}
&(1-t^2)^{n-k}=-\int_t^1(\rho^2-t^2)^{n-k-2}\cdot(n-k)(n-k-1)\cdot 4t^2\omega_n(\rho)\rho d\rho\\[5pt]
&\,+k\int_t^1(\rho^2-t^2)^{n-k-1}\cdot(n-k)\cdot 4\omega_n(\rho)\rho d\rho, \quad n=2,3,4,\ldots, \quad k=0,1,2,\ldots,n-2.
\end{split}
\]
This relation is equivalent to the Cauchy problem
\[\omega_n'(t)+\frac{2n}{t}\omega_n(t)=\frac{1}{t}, \quad \omega_n(1)=0.\]
Solving, we get
\[\omega_n(t)=\frac{1-t^{-2n}}{2n}.\]
Now we have the following relations between the orthonormal systems 
 $\{Q_j\}_1^\infty$ and $\{F_j\}_1^\infty$:

\begin{equation}\label{2.14}
\left \{\begin{split}
&\frac{1}{2\pi}\int_0^{2\pi} s_j \overline{Q_j(t,\theta)}d\theta=\frac{1}{2}A_{0j}\cdot lnt
,  \\[5pt]
&\frac{1}{\pi}\int_0^{2\pi}s_j \overline{Q_j(t,\theta)}\cdot \cos n\theta d\theta =-A_{nj}\cdot\dfrac{t^n-t^{-n}}{2n} ,\\[5pt]
&\frac{1}{\pi}\int_0^{2\pi}s_j \overline{Q_j(t,\theta)}\cdot \sin n\theta d\theta =-B_{nj}\cdot\dfrac{t^n-t^{-n}}{2n}, \quad n=1,2,\ldots
\end{split} \right.
\end{equation} 
Satisfying the Volterra criterion \eqref{2.10}, we obtained the relation \eqref{2.14}. By assumption $Q_j(t,\theta)$ from $L_2(\Omega)$. Then the integral with respect to $t$ on the left-hand sides of the system of equations \eqref{2.14} exists. However, for an arbitrary orthonormal system $\{F_j\}_1^\infty$, for $n=1,2,\ldots$, the integral on the right-hand sides of the system of equations \eqref{2.14} with respect to $t$ from $0$ to $1$ does not exist. This means that there are no orthonormal systems $\{F_j\}_1^\infty$ and $\{Q_j\}_1^\infty$ satisfying the equality \eqref{2.10}. This contradicts our assumption that there exists a Volterra correct restriction $L$. Thus, Theorem \ref{Theorem1} is proved.
\end{proof}

\begin{cor}\label{seq1} 
There does not exist a Volterra correct extension $L$ of the minimal operator $L_0$ generated by the Laplace operator \eqref{1.3} in a Hilbert space $L_2(\Omega)$, where $\Omega$ is the unit disk.
\end{cor}
\begin{proof} Suppose that there exists a Volterra correct extension $L$ of the minimal operator $L_0$. From $L_0\subset L$ it follows that $L^*\subset L_0^* =\widehat{L}$. The adjoint of a Volterra operator is a Volterra operator. Then we get a contradiction to Theorem \ref{Theorem1}. This completes the proof of Corollary~\ref{seq1}.
\end{proof}
In the author's work (see {\cite{Biyarov2}}), it was proved that there are no Volterra correct extensions or restrictions for the $m$--dimensional Laplace operator in $L_2(\Omega)$, where $\Omega$ is a bounded domain in $\mathbb R^m$ with a sufficiently smooth boundary, if the operator $K$ from representation \eqref{1.4} that it belongs to the Schatten class $\mathfrak{S}_{p}(L_2(\Omega))$ for $0<p\leq m/2$, where $m\geq2$.

It was noticed that in the case $m=1$ there exists many Volterra correct restrictions and extensions.

\begin{rem}\label{rem1}
If in \eqref{2.11} the function $u(r,\varphi)$ does not depend on the angle $\varphi$,  then we get one-dimensional equation 
\begin{equation} 
\label{2.15} 
\widehat{L}u \equiv -\Delta u=-u_{rr}-\frac{1}{r}u_r-=f(r), 
\end{equation}
in the weighted space $L_2 (r;0,1)$ with weight $r$.
Then the Volterra criterion \eqref{2.10} and the equation \eqref{2.14} determine the operator $K$ of the following form
\[Kf=\int_0^1 f(t)\ln t\cdot tdt.\]
To it corresponds to the correct restriction $L$ with domain $D(L)= \{u\in  W_2^1 (r;\,0,1) : u(0)=0\}$. Then the correct restriction $L$ is a Volterra, because its inverse operator 
\[u(r)=L^{-1}f=\int_0^r\ln \frac{t}{r}f(t)tdt,\]
is a Volterra in space $L_2 (r;\,0,1)$.
\end{rem}

\begin{rem}\label{rem2}
Theorem \ref{Theorem1} is true for every bounded simply connected domain in the plane, for which the Dirichlet problem is correct and there exists a conformal mapping onto the unit disk.
\end{rem}

\begin{rem}\label{rem3}
The generalization of Theorem \ref{Theorem1} to the $m$-dimensional ball (where $m\geqslant 3$) does not cause problems but it is cumbersome to write down.
\end{rem}

\vskip 0.5 cm

\begin{flushleft}
   Bazarkan Nuroldinovich Biyarov \\
   Department of Fundamental Mathematics\\
   L.N. Gumilyov Eurasian National University \\
   2, Satpayev St\\
   010008 Astana, Kazakhstan\\
   E-mail: bbiyarov@gmail.com
\end{flushleft} 

\end{document}